\pdfoutput=1
\documentclass[12pt]{article}
\usepackage{amsmath,amssymb,amsfonts}
\usepackage{tikz}
\begin{document}
\newtheorem{theorem}{Theorem}[section]
\newtheorem{corollary}[theorem]{Corollary}
\newtheorem{lemma}[theorem]{Lemma}
\newtheorem{remark}[theorem]{Remark}
\newtheorem{example}[theorem]{Example}
\newtheorem{proposition}[theorem]{Proposition}
\newtheorem{definition}[theorem]{Definition}
\def\emptyset{\varnothing}
\def\setminus{\smallsetminus}
\def\id{{\mathrm{id}}}
\def\G{{\mathcal{G}}}
\def\H{{\mathcal{H}}}
\def\C{{\mathbb{C}}}
\def\N{{\mathbb{N}}}
\def\Q{{\mathbb{Q}}}
\def\R{{\mathbb{R}}}
\def\Z{{\mathbb{Z}}}
\def\Path{{\mathrm{Path}}}
\def\Str{{\mathrm{Str}}}
\def\st{{\mathrm{st}}}
\def\tr{{\mathrm{tr}}}
\def\a{{\alpha}}
\def\be{{\beta}}
\def\de{{\delta}}
\def\e{{\varepsilon}}
\def\si{{\sigma}}
\def\la{{\lambda}}
\def\th{{\theta}}
\def\lan{{\langle}}
\def\ran{{\rangle}}
\def\isom{{\cong}}
\newcommand{\Hom}{\mathop{\mathrm{Hom}}\nolimits}
\def\qed{{\unskip\nobreak\hfil\penalty50
\hskip2em\hbox{}\nobreak\hfil$\square$
\parfillskip=0pt \finalhyphendemerits=0\par}\medskip}
\def\proof{\trivlist \item[\hskip \labelsep{\bf Proof.\ }]}
\def\endproof{\null\hfill\qed\endtrivlist\noindent}

\title{Projector Matrix Product Operators, Anyons and
Higher Relative Commutants of Subfactors}
\author{
{\sc Yasuyuki Kawahigashi}\\
{\small Graduate School of Mathematical Sciences}\\
{\small The University of Tokyo, Komaba, Tokyo, 153-8914, Japan}\\
{\small e-mail: {\tt yasuyuki@ms.u-tokyo.ac.jp}}
\\[0,40cm]
{\small Kavli IPMU (WPI), the University of Tokyo}\\
{\small 5--1--5 Kashiwanoha, Kashiwa, 277-8583, Japan}
\\[0,40cm]
{\small Trans-scale Quantum Science Institute}\\
{\small The University of Tokyo, Bunkyo-ku, Tokyo 113-0033, Japan}
\\[0,05cm]
{\small and}
\\[0,05cm]
{\small iTHEMS Research Group, RIKEN}\\
{\small 2-1 Hirosawa, Wako, Saitama 351-0198,Japan}}
\maketitle{}
\centerline{\sl Dedicated to the memory of Vaughan Jones}

\begin{abstract}
A bi-unitary connection in subfactor theory of Jones producing
a subfactor of finite depth gives
a 4-tensor appearing in a recent work of
Bultinck-Mari\"en-Williamson-\c Sahino\u glu-Haegeman-Verstraete
on two-dimensional topological order and anyons.  In their work, they
have a special projection called a projector matrix product operator.
We prove that the range of this projection of length $k$
is naturally identified with the $k$th higher relative
commutant of the subfactor arising from the bi-unitary connection.
This gives a further connection between two-dimensional topological 
order and subfactor theory.
\end{abstract}

\section{Introduction}

The Jones theory of subfactors \cite{J} in operator algebras
has found many profound relations to other topics in 
low-dimensional topology and mathematical physics.  Here we 
present a new connection between subfactor theory and
two-dimensional topological order.

Theory of topological phases of matter has recently caught
much attention both in mathematics and physics.  
A recent paper \cite{BMWSHV} on two-dimensional topological order,
tensor networks and anyons attracted much interest of
several researchers and this topic is closely related to
theory of topological quantum computation \cite{W}.
A certain operator $P^k$ on a finite dimensional Hilbert space
called a projector matrix product operator (PMPO)
\cite[Section 3]{BMWSHV}, arising from a certain 4-tensor which means
a finite family of complex numbers indexed with 4 indices,
plays a key role and its range is important in studies
of gapped Hamiltonians and 
projected entangled pair states (PEPS) as in 
\cite[Sections 4, 5]{BMWSHV} in connection to \cite{KLPG}, \cite{LW}.
The ranges of the projector matrix product operators $P^k$
give an increasing sequence of finite dimensional Hilbert spaces
indexed by $k$.  Our mail result, Theorem \ref{MT}, states that 
this space has a natural meaning
as the $k$th higher relative commutant of the subfactor arising
from the 4-tensor in the Jones theory, through repeated basic
constructions.
(See Figure \ref{mpoOa} for a matrix product operator $O_a^k$ which
is used in the definition of
$P^k=\displaystyle \sum_a\frac{d_a}{w}O_a^k$.)

We note that flatness of a field of strings in the sense 
of \cite[Theorems 11.15]{EK} is known to play
an important role in subfactor theory and it is also a key
notion in our main result.  (This flatness  was first introduced
by Ocneanu \cite{O2}.)    Recall that the tower of higher
relative commutants is one of the most important objects
in subfactor theory.

We have already seen a connection of the work \cite{BMWSHV}
to subfactor theory and the meaning of anyons there
in \cite{K3}, \cite{K4} and we now present a more
direct and deeper connection.  See \cite{LFHSV} for
another recent connection to theory of fusion categories,
which is also closely related to subfactor theory.
See \cite{K2} for more general relations among subfactor theory,
two-dimensional conformal field theory and tensor categories.

Many researchers work on a formulations based on fusion categories
in two-dimensional topological orders.
As shown in \cite[Chapter 12]{EK}, a fusion category framework
in terms of $6j$-symbols and one based on flat bi-unitary
connections (Definition \ref{DBC}) 
are equivalent.  A possible advantage of our framework
is that the size of numerical data is much smaller for bi-unitary
connections than $6j$-symbols and this could be more suited to
actual (numerical) computations.  We also treat non-flat bi-unitary 
connections simultaneously as flat bi-unitary connections
and this generality could cover a wider class of
examples.  (See Remark \ref{R1} on this point).

Recently we have much advance in operator algebraic classification
of gapped Hamiltonians on quantum spin chains \cite{Og} and we
see some formal similarities of mathematical structures there.
It would be interesting to exploit this possible connection.
For example, the range of a projector matrix product operator
should be a space of ground states in some appropriate sense and
this viewpoint is to be further explored.

This work was partially supported by 
JST CREST program JPMJCR18T6 and
Grants-in-Aid for Scientific Research 19H00640
and 19K21832.  I thank Ziyun Xu for comments improving the
exposition.

\section{A bi-unitary connection and a subfactor of finite depth}

In subfactor theory, finite bipartite graphs play an important role
as principal graphs of subfactors.  A vertex of a principal graph
represents an irreducible object in a certain tensor category and
an edge represents the dimension of a certain Hom space in such
a category.  We treat certain 4-tensors and their 4 wires are labeled
with edges of such finite bipartite graphs (and their slight
generalizations).  That is, a choice of four edges gives a complex
number and such an object is known as a bi-unitary connection
as in Definition \ref{DBC}
in subfactor theory.  We first prepare notations and conventions 
on bi-unitary connections as in \cite{AH},
\cite[Chapter 11]{EK}, \cite{K1}, \cite{K3} ,\cite{O1}, \cite{O2}.

We have four finite unoriented connected
bipartite graphs $\G,\G',\H,\H'$.  (These graphs are allowed to
have multiple edges between a pair of vertices.  The set of vertices
of each graph is divided into two classes, even add odd ones.)
The even vertices of $\G$ and $\H$ are identified and we write $V_0$ 
for the set of these vertices.
The odd vertices of $\H$ and $\G'$ are identified and we write $V_1$ 
for the set of these vertices.
The even vertices of $\G'$ and $\H'$ are identified and we write $V_2$ 
for the set of these vertices.
The odd vertices of $\G$ and $\H'$ are identified and we write $V_3$ 
for the set of these vertices.
They are depicted as in Figure \ref{four}.
We assume that all of the four graphs have more than one edges.

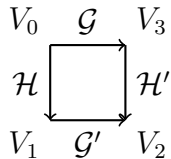
\begin{figure}[h]
\begin{center}
\begin{tikzpicture}
\draw [thick, ->] (1,1)--(2,1);
\draw [thick, ->] (1,2)--(2,2);
\draw [thick, ->] (1,2)--(1,1);
\draw [thick, ->] (2,2)--(2,1);
\draw (1,1.5)node[left]{$\H$};
\draw (2,1.5)node[right]{$\H'$};
\draw (1.5,1)node[below]{$\G'$};
\draw (1.5,2)node[above]{$\G$};
\draw (1,1)node[below left]{$V_1$};
\draw (1,2)node[above left]{$V_0$};
\draw (2,1)node[below right]{$V_2$};
\draw (2,2)node[above right]{$V_3$};
\end{tikzpicture}
\caption{Four graphs}
\label{four}
\end{center}
\end{figure}

Let $\Delta_{\G,xy}$ be the number of edges of $\G$ between
$x\in V_0$ and $y\in V_3$.
Let $\Delta_{\G',xy}$ be the number of edges of $\G$ between
$x\in V_1$ and $y\in V_2$.
Let $\Delta_{\H,xy}$ be the number of edges of $\H$ between
$x\in V_0$ and $y\in V_1$.
Let $\Delta_{\H',xy}$ be the number of edges of $\H'$ between
$x\in V_3$ and $y\in V_2$.  We assume that we have the following
identities for some positive numbers $\gamma_1,\gamma_2$.
For each vertex $x$, we have a positive number $\mu_x$.
We assume the following identities.  That is, for each of
$V_0, V_1, V_2, V_3$, the vector given by $\mu_x$ gives a
Perron-Frobenius eigenvector for the adjacency matrix of one 
of the four graphs, and the numbers $\gamma_1,\gamma_2$ are
the Perron-Frobenius eigenvalues of these matrices.
Since all the four graphs have more than one edge, we have
$\gamma_1,\gamma_2 > 1$.

\begin{align*}
\sum_x \Delta_{\G,xy} \mu_x=\gamma_1 \mu_y,\quad x\in V_0, y\in V_3,\\
\sum_y \Delta_{\G,xy} \mu_y=\gamma_1 \mu_y,\quad x\in V_0, y\in V_3,\\
\sum_x \Delta_{\G',xy} \mu_x=\gamma_1 \mu_y,\quad x\in V_1, y\in V_2,\\
\sum_y \Delta_{\G',xy} \mu_y=\gamma_1 \mu_y,\quad x\in V_1, y\in V_2,\\
\sum_x \Delta_{\H,xy} \mu_x=\gamma_2 \mu_y,\quad x\in V_0, y\in V_1,\\
\sum_y \Delta_{\H,xy} \mu_y=\gamma_2 \mu_y,\quad x\in V_0, y\in V_1,\\
\sum_x \Delta_{\H',xy} \mu_x=\gamma_2 \mu_y,\quad x\in V_3, y\in V_2,\\
\sum_y \Delta_{\H',xy} \mu_y=\gamma_2 \mu_y,\quad x\in V_3, y\in V_2,
\end{align*}

For an edge $\xi$ of one of the graphs $\G,\G',\H,\H'$, we
regard it oriented, and write $s(\xi)$ and $r(\xi)$ for the
source (starting vertex) and the range (ending vertex).
(Each graph is unoriented in the sense that for each edge
$\xi$, its reversed edge $\tilde \xi$ from $r(\xi)$ to 
$s(\xi)$ is also an edge of the graph and this reversing
map is bijective on the set of edges.)
Let $\xi_0,\xi_1,\xi_2,\xi_3$ be oriented edges of 
$\G,\H,\G',\H'$, respectively. 
If we have 
$s(\xi_0)=x_0\in V_0$,
$r(\xi_0)=x_1\in V_1$,
$s(\xi_1)=x_1\in V_1$,
$r(\xi_1)=x_2\in V_2$,
$s(\xi_2)=x_3\in V_3$,
$r(\xi_2)=x_2\in V_2$,
$s(\xi_3)=x_0\in V_0$, and
$r(\xi_3)=x_3\in V_3$,
then we call a combination of $\xi_i$ 
a {\em cell}, as in Figure \ref{cell}.

\begin{figure}[h]
\begin{center}
\begin{tikzpicture}
\draw [thick, ->] (1,1)--(2,1);
\draw [thick, ->] (1,2)--(2,2);
\draw [thick, ->] (1,2)--(1,1);
\draw [thick, ->] (2,2)--(2,1);
\draw (1,1.5)node[left]{$\xi_0$};
\draw (2,1.5)node[right]{$\xi_2$};
\draw (1.5,1)node[below]{$\xi_1$};
\draw (1.5,2)node[above]{$\xi_3$};
\draw (1,1)node[below left]{$x_1$};
\draw (1,2)node[above left]{$x_0$};
\draw (2,1)node[below right]{$x_2$};
\draw (2,2)node[above right]{$x_3$};
\end{tikzpicture}
\caption{A cell}
\label{cell}
\end{center}
\end{figure}

\begin{definition}\label{DC}{\rm
Assignment of a complex number to each cell is called
a {\em connection}.  We write $W$ for this map and
write $W$ within a cell to represent this number 
as in Figure \ref{connection}. 
}\end{definition}

Note that this setting is similar to an
interaction-round-a-face (IRF) model in theory of solvable
lattice models, where we also assign a complex number to 
each cell arising from one graph (rather than a combination 
of four graphs).

\begin{figure}[h]
\begin{center}
\begin{tikzpicture}
\draw [thick, ->] (1,1)--(2,1);
\draw [thick, ->] (1,2)--(2,2);
\draw [thick, ->] (1,2)--(1,1);
\draw [thick, ->] (2,2)--(2,1);
\draw (1.5,1.5)node{$W$};
\draw (1,1.5)node[left]{$\xi_0$};
\draw (2,1.5)node[right]{$\xi_2$};
\draw (1.5,1)node[below]{$\xi_1$};
\draw (1.5,2)node[above]{$\xi_3$};
\draw (1,1)node[below left]{$x_1$};
\draw (1,2)node[above left]{$x_0$};
\draw (2,1)node[below right]{$x_2$};
\draw (2,2)node[above right]{$x_3$};
\end{tikzpicture}
\caption{A connection value}
\label{connection}
\end{center}
\end{figure}
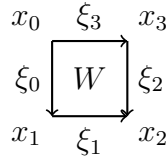

The {\em unitarity} axiom for $W$  is given in Figure \ref{unitarity}, 
where the bar on the right cell
denotes the complex conjugate of the connection value.

\begin{figure}[h]
\begin{center}
\begin{tikzpicture}
\draw [thick, ->] (2,1)--(3,1);
\draw [thick, ->] (2,2)--(3,2);
\draw [thick, ->] (2,2)--(2,1);
\draw [thick, ->] (3,2)--(3,1);
\draw (2.5,1.5)node{$W$};
\draw (2,1.5)node[left]{$\xi_1$};
\draw (3,1.5)node[right]{$\xi_3$};
\draw (2.5,1)node[below]{$\xi_2$};
\draw (2.5,2)node[above]{$\xi_4$};
\draw (2,1)node[below left]{$z$};
\draw (2,2)node[above left]{$x$};
\draw (3,1)node[below right]{$w$};
\draw (3,2)node[above right]{$y$};
\draw [thick, ->] (4,1)--(5,1);
\draw [thick, ->] (4,2)--(5,2);
\draw [thick, ->] (4,2)--(4,1);
\draw [thick, ->] (5,2)--(5,1);
\draw [thick] (3.7,2.8)--(5.3,2.8);
\draw (4.5,1.5)node{$W$};
\draw (4,1.5)node[left]{$\xi_1$};
\draw (5,1.5)node[right]{$\xi'_3$};
\draw (4.5,1)node[below]{$\xi_2$};
\draw (4.5,2)node[above]{$\xi'_4$};
\draw (4,1)node[below left]{$z$};
\draw (4,2)node[above left]{$x$};
\draw (5,1)node[below right]{$w$};
\draw (5,2)node[above right]{$y'$};
\draw (0.5,1.5)node{$\sum_{z,\xi_1,\xi_2}$};
\draw (7.5,1.5)node
{$=\displaystyle\delta_{y,y'}\delta_{\xi_3,\xi'_3}\delta_{\xi_4,\xi'_4}$};
\end{tikzpicture}
\caption{Unitarity}
\label{unitarity}
\end{center}
\end{figure}

We define a new connection $W'$ as
in Figure \ref{renormalization1} on the four graphs
$\tilde\G,\tilde\G',\H',\H$, where $\tilde\xi$ is the reversed
edge of $\xi$ from $r(\xi)$ to $s(\xi)$ and $\tilde\G$ is 
the reversed graph of $\G$ consisting of such reversed edges
as in Figure \ref{four2}.  We call this rule of giving a new
connection {\sl Renormalization}.

\begin{figure}[h]
\begin{center}
\begin{tikzpicture}
\draw [thick, ->] (2,1)--(3,1);
\draw [thick, ->] (2,2)--(3,2);
\draw [thick, ->] (2,2)--(2,1);
\draw [thick, ->] (3,2)--(3,1);
\draw (2.5,1.5)node{$W'$};
\draw (2,1.5)node[left]{$\xi_3$};
\draw (3,1.5)node[right]{$\xi_1$};
\draw (2.5,1)node[below]{$\tilde\xi_2$};
\draw (2.5,2)node[above]{$\tilde\xi_4$};
\draw (2,1)node[below left]{$w$};
\draw (2,2)node[above left]{$y$};
\draw (3,1)node[below right]{$z$};
\draw (3,2)node[above right]{$x$};
\draw [thick, ->] (6,1)--(7,1);
\draw [thick, ->] (6,2)--(7,2);
\draw [thick, ->] (6,2)--(6,1);
\draw [thick, ->] (7,2)--(7,1);
\draw [thick] (5.7,2.8)--(7.3,2.8);
\draw (6.5,1.5)node{$W$};
\draw (6,1.5)node[left]{$\xi_1$};
\draw (7,1.5)node[right]{$\xi_3$};
\draw (6.5,1)node[below]{$\xi_2$};
\draw (6.5,2)node[above]{$\xi_4$};
\draw (6,1)node[below left]{$z$};
\draw (6,2)node[above left]{$x$};
\draw (7,1)node[below right]{$w$};
\draw (7,2)node[above right]{$y$};
\draw (4.5,1.5)node{$\displaystyle=\sqrt{\frac{\mu_x\mu_w}{\mu_y\mu_z}}$};
\end{tikzpicture}
\caption{Renormalization (1)}
\label{renormalization1}
\end{center}
\end{figure}

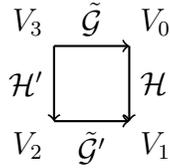
\begin{figure}[h]
\begin{center}
\begin{tikzpicture}
\draw [thick, ->] (1,1)--(2,1);
\draw [thick, ->] (1,2)--(2,2);
\draw [thick, ->] (1,2)--(1,1);
\draw [thick, ->] (2,2)--(2,1);
\draw (1,1.5)node[left]{$\H'$};
\draw (2,1.5)node[right]{$\H$};
\draw (1.5,1)node[below]{$\tilde\G'$};
\draw (1.5,2)node[above]{$\tilde\G$};
\draw (1,1)node[below left]{$V_2$};
\draw (1,2)node[above left]{$V_3$};
\draw (2,1)node[below right]{$V_1$};
\draw (2,2)node[above right]{$V_0$};
\end{tikzpicture}
\caption{Four graphs for $\tilde W$}
\label{four2}
\end{center}
\end{figure}

We now have the following definition of a bi-unitary connection.

\begin{definition}\label{DBC}{\rm
If unitarity holds for $W$ and $W'$, 
then we say that $W$ is a bi-unitarity connection
}\end{definition}

That is, bi-unitarity means that we have
unitarity for both the original connection $W$ and
the new connection $W'$ defined by 
Renormalization in Figure \ref{renormalization1}.
Roughly speaking, bi-unitarity means the connection
is ``doubly unitary'' for the original one and its 
reflection, but the connection
value should be adjusted for the reflection, and this 
adjustment up to normalization constants is given by
Renormalization, Figure \ref{renormalization1}.

Ocneanu and Haagerup found that
a bi-unitary connection characterizes a 
non-degenerate commuting squares of finite dimensional
$C^*$-algebras with a trace as in \cite[Section 11.2]{EK},

\begin{example}\label{dynkin}{\rm
A typical example of a bi-unitary connection is given as 
follows.  Fix one of the Dynkin diagrams $A_n, D_n, E_6, E_7, E_8$
and let $N$ be its Coxeter number.  Set all $\G,\G',\H,\H'$ to
be this bipartite graph so that $V_0$ and $V_2$ 
[resp. $V_1$ and $V_3$] give the
even [resp. odd] vertices of this graph 
and set both $\gamma_1,\gamma_2$ to be
$\displaystyle2\cos\frac{\pi}{N}$.  
We set $\varepsilon=\displaystyle\sqrt{-1}
\exp\left(\frac{\pi\sqrt{-1}}{2N}\right)$
We then have a
bi-unitary connection as in Figure \ref{dynkin2},
\cite[Figure 11.32]{EK}.

\begin{figure}[h]
\begin{center}
\begin{tikzpicture}
\draw [thick, ->] (2,1)--(3,1);
\draw [thick, ->] (2,2)--(3,2);
\draw [thick, ->] (2,2)--(2,1);
\draw [thick, ->] (3,2)--(3,1);
\draw (2.5,1.5)node{$W$};
\draw (2,1)node[below left]{$z$};
\draw (2,2)node[above left]{$x$};
\draw (3,1)node[below right]{$w$};
\draw (3,2)node[above right]{$y$};
\draw (5.5,1.5)node{$\displaystyle=\delta_{y,z}\varepsilon+
\sqrt{\frac{\mu_y\mu_z}{\mu_x\mu_w}}\delta_{x,w}\bar\varepsilon$};
\end{tikzpicture}
\caption{A bi-unitary connection on a Dynkin diagram}
\label{dynkin2}
\end{center}
\end{figure}

When the graph is $A_n$, this is related to the quantum group
$U_q(sl_2)$ with $q$ being a root of unity.
Also see \cite{P} for the corresponding IRF models.
}\end{example}

We assume this bi-unitarity for $W$ from now on.
(We do {\em not} assume flatness of $W$ in the sense
of \cite[Definition 11.16]{EK}.  If we have flatness with respect
to a vertex in $V_0$ and another in $V_2$, then this bi-unitary
connection gives a paragroup in the sense of Ocneanu 
\cite[Chapter 10]{EK}, when we would automatically have
$\G=\H$ and $\G'=\H'$.  In this sense, a bi-unitary connection
is a more general form of a paragroup.)

\begin{figure}[h]
\begin{center}
\begin{tikzpicture}
\draw [thick, ->] (2,1)--(3,1);
\draw [thick, ->] (2,2)--(3,2);
\draw [thick, ->] (2,2)--(2,1);
\draw [thick, ->] (3,2)--(3,1);
\draw (2.5,1.5)node{$\bar W$};
\draw (2,1.5)node[left]{$\tilde\xi_1$};
\draw (3,1.5)node[right]{$\tilde\xi_3$};
\draw (2.5,1)node[below]{$\xi_4$};
\draw (2.5,2)node[above]{$\xi_2$};
\draw (2,1)node[below left]{$x$};
\draw (2,2)node[above left]{$z$};
\draw (3,1)node[below right]{$y$};
\draw (3,2)node[above right]{$w$};
\draw [thick, ->] (6,1)--(7,1);
\draw [thick, ->] (6,2)--(7,2);
\draw [thick, ->] (6,2)--(6,1);
\draw [thick, ->] (7,2)--(7,1);
\draw [thick] (5.7,2.8)--(7.3,2.8);
\draw (6.5,1.5)node{$W$};
\draw (6,1.5)node[left]{$\xi_1$};
\draw (7,1.5)node[right]{$\xi_3$};
\draw (6.5,1)node[below]{$\xi_2$};
\draw (6.5,2)node[above]{$\xi_4$};
\draw (6,1)node[below left]{$z$};
\draw (6,2)node[above left]{$x$};
\draw (7,1)node[below right]{$w$};
\draw (7,2)node[above right]{$y$};
\draw (4.5,1.5)node{$\displaystyle=\sqrt{\frac{\mu_x\mu_w}{\mu_y\mu_z}}$};
\end{tikzpicture}
\caption{Renormalization (2)}
\label{renormalization2}
\end{center}
\end{figure}

\begin{figure}[h]
\begin{center}
\begin{tikzpicture}
\draw [thick, ->] (2,1)--(3,1);
\draw [thick, ->] (2,2)--(3,2);
\draw [thick, ->] (2,2)--(2,1);
\draw [thick, ->] (3,2)--(3,1);
\draw (2.5,1.5)node{$\bar W'$};
\draw (2,1.5)node[left]{$\tilde\xi_3$};
\draw (3,1.5)node[right]{$\tilde\xi_1$};
\draw (2.5,1)node[below]{$\tilde\xi_4$};
\draw (2.5,2)node[above]{$\tilde\xi_2$};
\draw (2,1)node[below left]{$y$};
\draw (2,2)node[above left]{$w$};
\draw (3,1)node[below right]{$x$};
\draw (3,2)node[above right]{$z$};
\draw [thick, ->] (5,1)--(6,1);
\draw [thick, ->] (5,2)--(6,2);
\draw [thick, ->] (5,2)--(5,1);
\draw [thick, ->] (6,2)--(6,1);
\draw (5.5,1.5)node{$W$};
\draw (5,1.5)node[left]{$\xi_1$};
\draw (6,1.5)node[right]{$\xi_3$};
\draw (5.5,1)node[below]{$\xi_2$};
\draw (5.5,2)node[above]{$\xi_4$};
\draw (5,1)node[below left]{$z$};
\draw (5,2)node[above left]{$x$};
\draw (6,1)node[below right]{$w$};
\draw (6,2)node[above right]{$y$};
\draw (4,1.5)node{$=$};
\end{tikzpicture}
\caption{Renormalization (3)}
\label{renormalization3}
\end{center}
\end{figure}

We also define new connections $\bar W$ and $\bar W'$ as in
Figures \ref{renormalization2} and \ref{renormalization3},
again on (partially) reversed graphs.
They are both bi-unitary connections automatically.

\begin{definition}{\rm
We define the value of another diagram  in the left hand side 
of Figure \ref{convention} as in
Figure \ref{convention}.
}\end{definition}

Note that we have the identity in Figure \ref{convention2} due to
Figures \ref{renormalization1} and \ref{convention}.

\begin{figure}[h]
\begin{center}
\begin{tikzpicture}
\draw [thick, ->] (3,1)--(2,1);
\draw [thick, ->] (3,2)--(2,2);
\draw [thick, ->] (2,2)--(2,1);
\draw [thick, ->] (3,2)--(3,1);
\draw (2.5,1.5)node{$W$};
\draw (2,1.5)node[left]{$\xi_3$};
\draw (3,1.5)node[right]{$\xi_1$};
\draw (2.5,1)node[below]{$\xi_2$};
\draw (2.5,2)node[above]{$\xi_4$};
\draw (2,1)node[below left]{$w$};
\draw (2,2)node[above left]{$y$};
\draw (3,1)node[below right]{$z$};
\draw (3,2)node[above right]{$x$};
\draw [thick, ->] (5,1)--(6,1);
\draw [thick, ->] (5,2)--(6,2);
\draw [thick, ->] (5,2)--(5,1);
\draw [thick, ->] (6,2)--(6,1);
\draw [thick] (4.7,2.8)--(6.3,2.8);
\draw (5.5,1.5)node{$W$};
\draw (5,1.5)node[left]{$\xi_1$};
\draw (6,1.5)node[right]{$\xi_3$};
\draw (5.5,1)node[below]{$\xi_2$};
\draw (5.5,2)node[above]{$\xi_4$};
\draw (5,1)node[below left]{$z$};
\draw (5,2)node[above left]{$x$};
\draw (6,1)node[below right]{$w$};
\draw (6,2)node[above right]{$y$};
\draw (4,1.5)node{$=$};
\end{tikzpicture}
\caption{Conjugate convention}
\label{convention}
\end{center}
\end{figure}

\begin{figure}[h]
\begin{center}
\begin{tikzpicture}
\draw [thick, ->] (3,1)--(2,1);
\draw [thick, ->] (3,2)--(2,2);
\draw [thick, ->] (2,2)--(2,1);
\draw [thick, ->] (3,2)--(3,1);
\draw (2.5,1.5)node{$W$};
\draw (2,1.5)node[left]{$\xi_1$};
\draw (3,1.5)node[right]{$\xi_3$};
\draw (2.5,1)node[below]{$\tilde\xi_2$};
\draw (2.5,2)node[above]{$\tilde\xi_4$};
\draw (2,1)node[below left]{$z$};
\draw (2,2)node[above left]{$x$};
\draw (3,1)node[below right]{$w$};
\draw (3,2)node[above right]{$y$};
\draw [thick, ->] (6,1)--(7,1);
\draw [thick, ->] (6,2)--(7,2);
\draw [thick, ->] (6,2)--(6,1);
\draw [thick, ->] (7,2)--(7,1);
\draw (6.5,1.5)node{$W$};
\draw (6,1.5)node[left]{$\xi_1$};
\draw (7,1.5)node[right]{$\xi_3$};
\draw (6.5,1)node[below]{$\xi_2$};
\draw (6.5,2)node[above]{$\xi_4$};
\draw (6,1)node[below left]{$z$};
\draw (6,2)node[above left]{$x$};
\draw (7,1)node[below right]{$w$};
\draw (7,2)node[above right]{$y$};
\draw (4.5,1.5)node{$\displaystyle=\sqrt{\frac{\mu_x\mu_w}{\mu_y\mu_z}}$};
\end{tikzpicture}
\caption{Renormalization convention}
\label{convention2}
\end{center}
\end{figure}

We fix any vertex in $V_0$ and write $*$ for this.  As in
\cite[Section 11.3]{EK}, we construct a double sequence of
finite dimensional $C^*$-algebras $\{A_{nk}\}_{n,k=0,1,\dots}$
starting from $*$
and hyperfinite II$_1$ factors $A_{\infty,k}$ and $A_{n,\infty}$,
using $W,W',\bar W,\bar W'$.
(Here our $\mu_*$ is not normalized to be $1$, so we use
$\mu_x/\mu_*$ to define a normalized trace on $A_{nk}$
as in \cite[page 554]{EK}.)
Then the we have
$[A_{\infty,1}:A_{\infty,0}]=\gamma_1^2$ and
$[A_{1,\infty}:A_{0,\infty}]=\gamma_2^2$
for the Jones index values as in 
\cite[Theorem 11.9]{EK}.  This construction is due to Ocneanu \cite{O1}.
We now assume that one of these
two subfactors has a {\em finite depth} in the sense of
\cite[Definition 9.41]{EK}.  Note that in this case, the other
subfactor also has a finite depth by a result of Sato,
\cite[Corollary 2.2]{S1}.
(This paper of Sato gave a positive solution 
to a question raised by Jones.)

Let $\tilde W$ be the (vertical) product of $W$ and $\bar W$ as in
Figure \ref{tildeW}.  That is, we multiply two connection values and
make a summation over all possible choices of $\xi_7$, like 
concatenation of tensors.
We make irreducible decomposition of powers of $\tilde W$.
As in \cite[Section 3]{AH}, this product and irreducible
decomposition correspond to the relative tensor product
and irreducible decompositions of $A_{0,\infty}$-$A_{0,\infty}$
bimodules arising from the subfactor $A_{0,\infty}\subset A_{1,\infty}$.
(These bimodules are also understood in terms of sectors as in
\cite{L}.)

\begin{figure}[h]
\begin{center}
\begin{tikzpicture}
\draw [thick, ->] (1,2.5)--(2,2.5);
\draw [thick, ->] (1,0.5)--(2,0.5);
\draw [thick, ->] (4,1)--(5,1);
\draw [thick, ->] (4,2)--(5,2);
\draw [thick, ->] (6.5,1)--(7.5,1);
\draw [thick, ->] (6.5,2)--(7.5,2);
\draw [thick, ->] (1,2.5)--(1,1.5);
\draw [thick, ->] (1,1.5)--(1,0.5);
\draw [thick, ->] (2,2.5)--(2,1.5);
\draw [thick, ->] (2,1.5)--(2,0.5);
\draw [thick, ->] (4,2)--(4,1);
\draw [thick, ->] (5,2)--(5,1);
\draw [thick, ->] (6.5,2)--(6.5,1);
\draw [thick, ->] (7.5,2)--(7.5,1);
\draw (1.5,1.5)node{$\tilde W$};
\draw (4.5,1.5)node{$W$};
\draw (7,1.5)node{$\bar W$};
\draw (1,2)node[left]{$\xi_1$};
\draw (1,1)node[left]{$\xi_2$};
\draw (1.5,0.5)node[below]{$\xi_3$};
\draw (2,1)node[right]{$\xi_4$};
\draw (2,2)node[right]{$\xi_5$};
\draw (1.5,2.5)node[above]{$\xi_6$};
\draw (4,1.5)node[left]{$\xi_1$};
\draw (4.5,1)node[below]{$\xi_7$};
\draw (5,1.5)node[right]{$\xi_5$};
\draw (4.5,2)node[above]{$\xi_6$};
\draw (6.5,1.5)node[left]{$\xi_2$};
\draw (7,1)node[below]{$\xi_3$};
\draw (7.5,1.5)node[right]{$\xi_4$};
\draw (7,2)node[above]{$\xi_7$};
\draw (5.75,1.5)node{$\times$};
\draw (2.9,1.3)node{$\displaystyle=\sum_{\xi_7}$};
\end{tikzpicture}
\caption{The product connection $\tilde W$}
\label{tildeW}
\end{center}
\end{figure}

Let $\{W_a\}_{a\in V}$ be the set of representative of
irreducible bi-unitary connections, up to equivalence,
appearing in the
irreducible decompositions of the powers of $\tilde W$.
(See \cite[Section 3]{AH} for the definition of equivalence
of connections.  This corresponds to an isomorphism of
bimodules.)  The finite depth assumption exactly
means that the set $V$ is finite.
Each $a$ corresponds to an irreducible $A_{0,\infty}$-$A_{0,\infty}$
bimodules arising from the subfactor $A_{0,\infty}\subset A_{1,\infty}$.
Each $a$ thus also corresponds to an even vertex of the 
principal graph of the subfactor $A_{0,\infty}\subset A_{1,\infty}$.
Note that the horizontal graph of each $W_a$ is always
the original graph $\G$.

Let $d_a$ be the Perron-Frobenius eigenvalue of the vertical
graph corresponding to the bi-unitary connection $W_a$.  This
is equal to the dimension of the bimodule corresponding to $W_a$
as in \cite[Section 3]{AH}.
We define $w=\sum_{a\in V} d_a^2$, which is sometimes called the
{\em global index} of the subfactor $A_{0,\infty}\subset A_{1,\infty}$.
The original Perron-Frobenius vector $(\mu_x)_{x\in V_0}$ of $\H$ 
is unique up scalar.  We now normalize this vector so that we have
$\sum_{x\in V_0} \mu_x^2=w$.  Note that the Perron-Frobenius 
vector $(\mu_x)_{x\in V_0}$ is also an eigenvector for the vertical
graph corresponding to each bi-unitary connection $W_a$.

Let $M_{xa}^y$ be the number of vertical edges 
with vertex $x\in V_0$ at the upper left corner and  $y\in V_0$
at the lower left corner
for the connection $W_a$.   Note that the
Perron-Frobenius eigenvalue property gives 
$\sum_{y\in V_0} M_{xa}^y \mu_y=d_a \mu_x$.
For $a,b,c\in V$, let $N_{ab}^c$ be the
multiplicity of $W_c$ in the irreducible decomposition of the
product $W_a W_b$.   This is also the structure constant of
relative tensor products of the corresponding bimodules 
over $A_{0,\infty}$.

We define $\bar a$ to be $b\in V$ so that
$\bar W_a$ is equivalent to $W_b$.
We have $M_{xa}^y=M_{y\bar a}^x$ by the Frobenius reciprocity,
\cite[Section 9.8]{EK}.  The $A_{0,\infty}$-$A_{0,\infty}$
bimodule corresponding to $\bar a$ is contragredient to the one
corresponding to $a$ by a result in \cite[Page 17]{AH}.

Finally, we recall the following elementary lemma about a conditional
expectation in the string algebra.
(See \cite[Definitions 11.1, 11.4]{EK} for string algebras and a
trace there.)

\begin{lemma}\label{L1}
Let $A=\C\subset B\subset C$ be an increasing sequence of string algebras
of length $0,1,2$ on a Bratteli diagram.  We write $*$ for the initial
vertex of the Bratteli diagram corresponding to $A=\C$.
We fix a faithful trace on $C$.
The conditional expectation $E$ from $C$ onto $B'\cap C$ is given as
follows.

Let $\xi_1,\xi_2$ be edges of the Bratteli diagram corresponding
to $A\subset B$, $\eta_1,\eta_2$ be edges of the one 
corresponding to $B\subset C$.  Assume $r(\xi_1)=s(\eta_1)$,
$r(\xi_2)=s(\eta_2)$, $r(\eta_1)=r(\eta_2)$.
We then have 
\[
E((\xi_1\cdot\eta_1,\xi_2\cdot\eta_2))=
\delta_{\xi_1,\xi_2} \frac{1}{K_{r(\xi_1)}}
\sum_{\xi}(\xi\cdot\eta_1,\xi\cdot\eta_2),
\]
where $K_{r(\xi_1)}$ is the number of edges from $*$
to $r(\xi_1)$ on the Bratteli diagram corresponding
to $A\subset B$.
\end{lemma}

\begin{proof}
We have this identity by a direct computation.
\end{proof}

\section{A 4-tensor and a projector matrix product operator}

We define projector matrix product operators \cite[Section 3.1]{BMWSHV},
which was originally defined in terms of 4-tensors, with language
of bi-unitary connections in the previous Section.

We define a 4-tensor $a$ as in Figure \ref{tensor} and
\cite[Figure 11]{K3}.  Note that we have a horizontal
concatenation of the connections $W_a$ and $W'_a$ here,
since we have considered only
{\em symmetric} bi-unitary connections in \cite[Section 2]{K3}
while we do not assume this symmetric condition here.
(If we have $s(\xi_1)\neq s(\xi_6)$, then the value of the 4-tensor
is set to be 0.  Similarly, if the edges do not make a cell for
one of the two squares, the value of the 4-tensor is 0.)

\begin{figure}[h]
\begin{center}
\begin{tikzpicture}
\draw [thick] (2.5,1)--(2.5,1.3);
\draw [thick] (2.5,2)--(2.5,1.7);
\draw [thick] (2,1.5)--(2.3,1.5);
\draw [thick] (3,1.5)--(2.7,1.5);
\draw (2.5,1.5) circle (0.2);
\draw (2.5,1.5)node{$a$};
\draw (2,1.5)node[left]{$\xi_1$};
\draw (3,1.5)node[right]{$\xi_4$};
\draw (2.5,1)node[below]{$\xi_2\cdot\xi_3$};
\draw (2.5,2)node[above]{$\xi_6\cdot\xi_5$};
\draw [thick, ->] (6,1)--(7,1);
\draw [thick, ->] (6,2)--(7,2);
\draw [thick, ->] (6,2)--(6,1);
\draw [thick, ->] (7,2)--(7,1);
\draw [thick, ->] (7,1)--(8,1);
\draw [thick, ->] (7,2)--(8,2);
\draw [thick, ->] (8,2)--(8,1);
\draw (6.5,1.5)node{$W_a$};
\draw (7.5,1.5)node{$W'_a$};
\draw (6,1.5)node[left]{$\xi_1$};
\draw (6.5,1)node[below]{$\xi_2$};
\draw (7.5,1)node[below]{$\xi_3$};
\draw (8,1.5)node[right]{$\xi_4$};
\draw (7.5,2)node[above]{$\xi_5$};
\draw (6.5,2)node[above]{$\xi_6$};
\draw (6,1)node[below left]{$z$};
\draw (6,2)node[above left]{$x$};
\draw (8,1)node[below right]{$w$};
\draw (8,2)node[above right]{$y$};
\draw (4.5,1.5)node{$\displaystyle=\sqrt[4]{\frac{\mu_x\mu_w}{\mu_y\mu_z}}$};
\end{tikzpicture}
\caption{The $4$-tensor $a$ and the connection $W_a$}
\label{tensor}
\end{center}
\end{figure}

\begin{remark}\label{R1}{\rm
When we concatenate edges $\xi_1,\xi_2,\dots,\xi_k$ taken from
the horizontal graph of $W_a$, we impose the condition
$r(\xi_m)=s(\xi_{m+1})$ for $m=1,2,\dots,k-1$.  In the
4-tensor setting, we do not impose such a condition
for concatenation of edges, but 
this difference does not cause any problem here.
If we have $r(\xi_m)\neq s(\xi_{m+1})$, the path
$\xi_1\dots\xi_2\cdots\xi_k$ is mapped to zero by any
matrix product operator and we can ignore this path, since
we are interested in the range of a matrix product operator.
}\end{remark}

Fix a positive integer $k$. 
Let $\Path^{2k}(\G)$ be
the $\C$-vector space with a basis consisting of
paths of length $2k$ on 
$\G$ starting at an even vertex of $\G$.
We define a matrix product operator $O_{a,x}^{k,y}$ from
$\Path_{x,x}^{2k}(\G)$ to $\Path_{y,y}^{2k}(\G)$, where
$\Path_{x,x}^{2k}(\G)$ is a $\C$-linear space spanned by 
paths of length $2k$ starting from $x$ to $x$ on $\G$, as in Figure \ref{Oaxy},
where $\xi_1$ and $\xi_2$ have length $k$ each.

\begin{figure}[h]
\begin{center}
\begin{tikzpicture}
\draw [thick, ->] (4,1)--(7,1);
\draw [thick, ->] (7,1)--(10,1);
\draw [thick, ->] (4,2)--(7,2);
\draw [thick, ->] (7,2)--(10,2);
\draw [thick, ->] (4,2)--(4,1);
\draw [thick, ->] (4,2)--(4,1);
\draw [thick, ->] (5,2)--(5,1);
\draw [thick, ->] (6,2)--(6,1);
\draw [thick, ->] (7,2)--(7,1);
\draw [thick, ->] (8,2)--(8,1);
\draw [thick, ->] (9,2)--(9,1);
\draw [thick, ->] (10,2)--(10,1);
\draw (4.5,1.5)node{$W_a$};
\draw (5.5,1.5)node{$W'_a$};
\draw (6.5,1.5)node{$\cdots$};
\draw (7.5,1.5)node{$\cdots$};
\draw (8.5,1.5)node{$W'_a$};
\draw (9.5,1.5)node{$W_a$};
\draw (5.5,1)node[below]{$\eta_1$};
\draw (8.5,1)node[below]{$\eta_2$};
\draw (5.5,2)node[above]{$\xi_1$};
\draw (8.5,2)node[above]{$\xi_2$};
\draw (11.5,1.5)node{$\eta_1\cdot\eta_2$};
\draw (4,1.5)node[left]{$\zeta$};
\draw (10,1.5)node[right]{$\zeta$};
\draw (4,2)node[above left]{$x$};
\draw (4,1)node[below left]{$y$};
\draw (10,2)node[above right]{$x$};
\draw (10,1)node[below right]{$y$};
\draw (1.5,1.3)node{$O_{a,x}^{k,y}(\xi_1\cdot\xi_2)
=\displaystyle\sum_{\zeta,\eta_1,\eta_2}$};
\end{tikzpicture}
\caption{The operator $O_{a,x}^{k,y}$}
\label{Oaxy}
\end{center}
\end{figure}

We next define a matrix product operator $O_a^k$ by
\[
O_a^k (\bigoplus_x \xi_x)=\bigoplus_y \sum_x O_{a,x}^{k,y} \xi_x,
\]
where $\xi_x\in \Path_{x,x}^{2k}(\G)$.
Note that this is the same as the matrix product operator 
$O_a^k$ defined by Figure \ref{mpoOa} as in 
\cite[Section 3.2]{BMWSHV}.  We have different normalization
convention for the tensor $a$ and the connection $W_a$ as
in Figure \ref{tensor}, but these coefficients cancel out
due to the horizontal periodicity of the picture.
(Remark \ref{R1} again applies here about the domains of 
the two operators $O_a^k$.)

\begin{figure}[h]
\begin{center}
\begin{tikzpicture}
\draw [thick] (2.5,1)--(2.5,1.3);
\draw [thick] (2.5,2)--(2.5,1.7);
\draw [thick] (3.5,1)--(3.5,1.3);
\draw [thick] (3.5,2)--(3.5,1.7);
\draw [thick] (5.5,1)--(5.5,1.3);
\draw [thick] (5.5,2)--(5.5,1.7);
\draw [thick] (2,1.5)--(2.3,1.5);
\draw [thick] (2.7,1.5)--(3.3,1.5);
\draw [thick] (3.7,1.5)--(4.1,1.5);
\draw [thick] (4.9,1.5)--(5.3,1.5);
\draw [thick] (5.7,1.5)--(6,1.5);
\draw (2.5,1.5) circle (0.2);
\draw (3.5,1.5) circle (0.2);
\draw (5.5,1.5) circle (0.2);
\draw (2.5,1.5)node{$a$};
\draw (3.5,1.5)node{$a$};
\draw (4.5,1.5)node{$\cdots$};
\draw (5.5,1.5)node{$a$};
\draw (2.5,1)node[below]{$\eta_1$};
\draw (2.5,2)node[above]{$\xi_1$};
\draw (3.5,1)node[below]{$\eta_2$};
\draw (3.5,2)node[above]{$\xi_2$};
\draw (5.5,1)node[below]{$\eta_k$};
\draw (5.5,2)node[above]{$\xi_k$};
\draw (0.1,1.2)node{$\displaystyle\sum_{\xi_1,\xi_2,\dots,\xi_k,
\eta_1,\eta_2,\dots,\eta_k}$};
\draw (8.8,1.5)node{$\mid \xi_1\xi_2\cdots\xi_k\rangle
\langle\eta_1\eta_2\cdots\eta_k\mid$};
\draw [thick] (2,1.5) arc (90:270:0.5);
\draw [thick] (6.5,1) arc (0:90:0.5);
\draw [thick] (6,0.5) arc (270:360:0.5);
\draw [thick] (2,0.5)--(6,0.5);
\end{tikzpicture}
\caption{The matrix product operator $O_a^k$}
\label{mpoOa}
\end{center}
\end{figure}

We then have $O_a^k O_b^k=\sum_c N_{ab}^c O_c^k$.
We further define a projector matrix product operator 
$P^k=\displaystyle \sum_a\frac{d_a}{w}O_a^k$
as in \cite[Section 3.1]{BMWSHV}.  (This is a projection
as shown there.)

For a path $\xi_1\cdot\xi_2$ with $r(\xi_1)=s(\xi_2)$ and
$|\xi_1|=|\xi_2|=k$, we define
$\Phi^k(\xi_1\cdot\xi_2)=\displaystyle
\sqrt{\frac{\mu_{s(\xi_1)}}{\mu_{r(\xi_1)}}}(\xi_1,\tilde\xi_2)$, which
is a map from $\Path^{2k}(\G)$ to $B_k$, where 
$\tilde\xi_2$ is the reversed path of $\xi_2$,
$B_k=\bigoplus_x \Str_x^k(\G)$ and
$\Str_x^k(\G)$ is the string algebra on $\G$ with length
$k$ starting at a vertex $x\in V_0$ of $\G$.
(See \cite[Definitions 11.1, 11.4]{EK} for string algebras.)

We define a matrix product operator $\tilde O_{a,x}^{k,y}$
from $\Str_x^k(\G)$  to $\Str_y^k(\G)$
as in Figure \ref{tOaxy}.  

\begin{figure}[h]
\begin{center}
\begin{tikzpicture}
\draw [thick, ->] (4,1)--(7,1);
\draw [thick, ->] (10,1)--(7,1);
\draw [thick, ->] (4,2)--(7,2);
\draw [thick, ->] (10,2)--(7,2);
\draw [thick, ->] (4,2)--(4,1);
\draw [thick, ->] (4,2)--(4,1);
\draw [thick, ->] (5,2)--(5,1);
\draw [thick, ->] (6,2)--(6,1);
\draw [thick, ->] (7,2)--(7,1);
\draw [thick, ->] (8,2)--(8,1);
\draw [thick, ->] (9,2)--(9,1);
\draw [thick, ->] (10,2)--(10,1);
\draw (4.5,1.5)node{$W_a$};
\draw (5.5,1.5)node{$W'_a$};
\draw (6.5,1.5)node{$\cdots$};
\draw (7.5,1.5)node{$\cdots$};
\draw (8.5,1.5)node{$W'_a$};
\draw (9.5,1.5)node{$W_a$};
\draw (5.5,1)node[below]{$\eta_1$};
\draw (8.5,1)node[below]{$\eta_2$};
\draw (5.5,2)node[above]{$\xi_1$};
\draw (8.5,2)node[above]{$\xi_2$};
\draw (11.5,1.5)node{$(\eta_1,\eta_2)$};
\draw (4,1.5)node[left]{$\zeta$};
\draw (10,1.5)node[right]{$\zeta$};
\draw (4,2)node[above left]{$x$};
\draw (4,1)node[below left]{$y$};
\draw (10,2)node[above right]{$x$};
\draw (10,1)node[below right]{$y$};
\draw (1.5,1.3)node{$\tilde O_{a,x}^{k,y}((\xi_1,\xi_2))
=\displaystyle\sum_{\zeta,\eta_1,\eta_2}$};
\end{tikzpicture}
\caption{The operator $\tilde O_{a,x}^{k,y}$}
\label{tOaxy}
\end{center}
\end{figure}

We next define a matrix product operator $\tilde O_a^k$ by
\[
\tilde O_a^k (\bigoplus_x \xi_x)=
\bigoplus_y \sum_x \tilde O_{a,x}^{k,y} \xi_x,
\]
where $\xi_x\in \Str_x^k(\G)$.
We again have 
$\tilde O_a^k \tilde O_b^k=\sum_{c\in V} N_{ab}^c \tilde O_c^k$.
We further define a projector matrix product operator 
$\tilde P^k=\displaystyle \sum_{a\in V}\frac{d_a}{w}\tilde O_a^k$
again as in \cite[Section 3.1]{BMWSHV}.

We then have
$\Phi^k O_a^k=\tilde O_a^k \Phi^k$ because of
the normalization in Figure \ref{convention2} and 
$(\tilde P^k)^2=\tilde P^k$ for the same reason as
$(P^k)^2=P^k$.

Each $\Str_x^k(\G)$  has a standard normalized trace
$\tr_x$ as in \cite[page 554]{EK}.  We set $\tr(\sigma)=
\displaystyle \sum_{x\in V_0} \frac{\mu_x^2}{w}\tr_x(\sigma_x)$
for $\sigma=\bigoplus_{x\in V_0} \sigma_x\in 
\bigoplus_{x\in V_0}\Str_x^k(\G)$.
We let $\|\sigma\|_{\st, 2}=\sqrt{\tr(\si^*\si)}$ for 
$\si\in B_k$.

Let $C$ be the maximum of the number of $x\in V_0$, the number
of $a\in V$, $\|\tilde O_a^k\|$ over all $a\in V$
and $\displaystyle\frac{d_a \mu_x}{w \mu_y}$ over 
all $a\in V$, $x,y\in V_0$.  Here the norm
$\|\tilde O_a^k\|$ is the operator norm on $B_k$ with
respect to $\|\cdot\|_{\st,2}$.  Note that we have $C\ge 1$.

Let $K_x^n$ be the number of paths from $*$ to $x$ on $\H$
of length $2n$.  Let $\a_n=\sqrt{\sum_x (K_x^n)^2}$,
and $\kappa_x^n=K_x^n/\a_n$. By the Perron-Frobenius
theorem, we have $\kappa_x^n\to \mu_x/\sqrt w$ as $n\to\infty$
for all $x\in V_0$.

For a positive integer $n$, let $\tilde W^n\isom \sum_a L_a^n W_a$,
$\be_n=\sqrt{\sum_a (L_a^n)^2}$
and $\la_a^n=L_a^n/\be_n$.  By the Perron-Frobenius
theorem again, we have $\la_a^n\to d_a/\sqrt w$ as $n\to\infty$
for all $a\in V$.

We recall the following elementary lemma.

\begin{lemma}\label{L2}
Let $M$ be a von Neumann algebra with a normalized trace
$\tr$ and $P$ be its subalgebra.  For $\si\in M$ and
$\e<1$, if we have
$| \|\si\|_2-\|E_P(\si)\|_2| <\e \|\si\|_2$, then we have
$\|\si-E_P(\si)\|_2<\sqrt2\sqrt\e\|\si\|_2$.
\end{lemma}

\begin{proof}
Since $\|E_P(\si)\|_2>(1-\e)\|\si\|_2$ and 
$\|\si\|_2^2=\|E_P(\si)\|^2+\|\si-E_P(\si)\|^2$,
we have the conclusion.
\end{proof}

With these preparations, we are going to prove the following
main result of this paper.

\begin{theorem}\label{MT}
The range of the projector matrix product operator $P^k$ of length
$k$ is naturally identified with the $k$th higher relative commutant
$A'_{\infty,0}\cap A_{\infty,k}$ for the subfactor
$A_{\infty,0}\subset A_{\infty,1}$ arising from the original 
connection $W$.
\end{theorem}

\begin{proof}
Note that the map $\Phi^k$ gives a linear isomorphism from
the range of $P^k$ to that of $\tilde P^k$ in $B_k$.

We first construct a linear isomorphism $\Delta$ from
$A'_{\infty,0}\cap A_{\infty,k}$ to the range of 
$\tilde P^k$.  By \cite[Theorem 11.15]{EK}, an arbitrary element in
$A'_{\infty,0}\cap A_{\infty,k}$  is given by a flat
field $\bigoplus_x \si_x \in B_k$ and identified with
$\si_* \in A_{0,k}$.

We define an operator $T_{a,x,\zeta_1,\zeta_2}^{k,y}$
from $\Str_x^k(\G)$ to $\Str_y^k(\G)$ as in
Figure \ref{Taxy}.

\begin{figure}[h]
\begin{center}
\begin{tikzpicture}
\draw [thick, ->] (4,1)--(7,1);
\draw [thick, ->] (10,1)--(7,1);
\draw [thick, ->] (4,2)--(7,2);
\draw [thick, ->] (10,2)--(7,2);
\draw [thick, ->] (4,2)--(4,1);
\draw [thick, ->] (4,2)--(4,1);
\draw [thick, ->] (5,2)--(5,1);
\draw [thick, ->] (6,2)--(6,1);
\draw [thick, ->] (7,2)--(7,1);
\draw [thick, ->] (8,2)--(8,1);
\draw [thick, ->] (9,2)--(9,1);
\draw [thick, ->] (10,2)--(10,1);
\draw (4.5,1.5)node{$W_a$};
\draw (5.5,1.5)node{$W'_a$};
\draw (6.5,1.5)node{$\cdots$};
\draw (7.5,1.5)node{$\cdots$};
\draw (8.5,1.5)node{$W'_a$};
\draw (9.5,1.5)node{$W_a$};
\draw (5.5,1)node[below]{$\eta_1$};
\draw (8.5,1)node[below]{$\eta_2$};
\draw (5.5,2)node[above]{$\xi_1$};
\draw (8.5,2)node[above]{$\xi_2$};
\draw (11.5,1.5)node{$(\eta_1,\eta_2)$};
\draw (4,1.5)node[left]{$\zeta_1$};
\draw (10,1.5)node[right]{$\zeta_2$};
\draw (4,2)node[above left]{$x$};
\draw (4,1)node[below left]{$y$};
\draw (10,2)node[above right]{$x$};
\draw (10,1)node[below right]{$y$};
\draw (1.2,1.3)node{$T_{a,x,\zeta_1,\zeta_2}^{k,y}((\xi_1,\xi_2))
=\displaystyle\sum_{\eta_1,\eta_2}$};
\end{tikzpicture}
\caption{The operator $T_{a,x,\zeta_1,\zeta_2}^{k,y}$}
\label{Taxy}
\end{center}
\end{figure}

Then flatness of the field \cite[Theorems 11.15]{EK}
gives the equality
$T_{a,x,\zeta_1,\zeta_2}^{k,y}(\si_x)=
\delta_{\zeta_1,\zeta_2}\si_y$.
(This holds as in \cite[Figure 11.19]{EK}.  Though flatness
of the bi-unitary connection is not assumed here, flatness
of the fields works instead.)
This implies that
$\tilde O_a^k \si_x= \bigoplus_y M_{xa}^y \si_y$.
Note we have 
\[
\sum_{x\in V_0,a\in V} d_a \mu_x M_{xa}^y=
\sum_{a\in V} d_a \sum_{x\in V_0} \mu_x M_{y\bar a}^x=
\sum_{a\in V} d_a^2 \mu_y =w\mu_y.
\]
We then have 
\[
\tilde P^k(\bigoplus_{x\in V_0} \mu_x \si_x)=
\sum_{x\in V_0}\frac{d_a}{w}\tilde O_a^k \mu_x \si_x=
\bigoplus_{y\in V_0} \sum_{x\in V_0,a\in V}
\frac{d_a}{w} \mu_x M_{xa}^y \si_y
=\bigoplus_{y\in V_0} \mu_y \si_y
\] 
so the map $\Delta$ assigning
$\bigoplus_{x\in V_0} \mu_x \si_x$ to $\si_*$ gives a linear injection
from $A'_{\infty,0}\cap A_{\infty,k}$ 
to the range of $\tilde P^k$ in $B_k$.

We next construct an injective linear map for the
converse direction.  Let 
$\bigoplus_{x\in V_0} \mu_x \si_x$ be in
the range of $\tilde P^k$ in $B_k$.
For a positive integer $n$, we set 
\[
\si^{(n)}=\sum_{x\in V_0} \sum_\xi (\xi,\xi)\cdot\si_x
\in A_{2n,k},
\]
where $\xi$ gives all paths from $*$ to $x$ on $\H$
with length $2n$.  We assume that $n$ is sufficiently large 
so that the numbers $K_x^n$ are all nonzero.

Suppose that we have the following three estimates for sufficiently
small $\e>0$.

\begin{equation}\label{gamma}
\frac{1-\e}{\sqrt w} < \alpha_p \gamma_2^{-2p}
<\frac{1+\e}{\sqrt w},\quad
\hbox{for all $p\ge n$}
\end{equation}

\begin{equation}\label{kappa}
(1-\e)\frac{\mu_x}{\sqrt w} < \kappa_x^n
<(1+\e)\frac{\mu_x}{\sqrt w},\quad
\hbox{for all $x\in V_0$}
\end{equation}

\begin{equation}\label{lambda}
(1-\e)\frac{d_a}{\sqrt w} < \la_a^m
<(1+\e)\frac{d_a}{\sqrt w},\quad
\hbox{for all $V\in a$}
\end{equation}

A computation shows that
$E_{A'_{2n+2m,0}\cap A_{2n+2m,k}}(\si^{(n)})$ is equal to
\[
\sum_{y\in V_0} \sum_{\xi} (\xi,\xi)\cdot 
\sum_{a\in V,x\in V_0}\frac{K_x^n}{K_y^{n+m}}
L_a^m \tilde O_{a,x}^{k,y}(\si_x)
\] 
by Lemma \ref{L1}, whrere $\xi$ 
gives all paths from $*$ to $y$ on $\H$ with length $2(n+m)$,
since we have $\tilde W\isom\sum_{a\in V} L_a^m W_a$.
Here for large $n$ and $m$, $K_x^n$ is almost equal to
$\displaystyle\frac{\a_n\mu_x}{\sqrt w}$,
$L_a^m$ is almost equal to
$\displaystyle\frac{\be_m d_a}{\sqrt w}$,
and $K_y^{n+m}$ is almost equal to
\begin{align*}
\sum_{a\in V,x\in V_0}\frac{\a_n\mu_x}{\sqrt w}
\frac{\be_m d_a}{\sqrt w}M_{xa}^y
&=\frac{\a_n\be_m}{w}\sum_{a\in V,x\in V_0}d_a \mu_x M_{y\bar a}^x\\
&=\frac{\a_n\be_m}{w}\sum_{a\in V} d_a^2 \mu_y\\
&=\a_n\be_m\mu_y.
\end{align*}
If we had exact equalities for all these three pairs, then we would have
\begin{align*}
\bigoplus_{y\in V_0} \sum_{a\in V,x\in V_0}\frac{K_x^n}{K_y^{n+m}}
L_a^m \tilde O_{a,x}^{k,y}(\si_x)
&=\bigoplus_{y\in V_0} \sum_{a\in V, x\in V_0}
\frac{\mu_x}{\sqrt w \be_m \mu_y}\frac{\be_m d_a}{\sqrt w}
\tilde O_{a,x}^{k,y}(\si_x)\\
&=\bigoplus_{y\in V_0} \sum_{a\in V, x\in V_0} 
\frac{d_a \mu_x}{w\mu_y}\tilde O_{a,x}^{k,y}(\si_x)\\
&=\bigoplus_{y\in V_0} \si_y,
\end{align*}
where the last equality would follow from
\[
\sum_{a\in V}\frac{d_a}{w}\tilde O_a^k
 (\bigoplus_{x\in V_0} \mu_x \sigma_x)
=\tilde P_k(\bigoplus_{x\in V_0} \mu_x \sigma_x)
=\bigoplus_{y\in V_0} \mu_y \sigma_y,
\]
so $E_{A'_{2n+2m,0}\cap A_{2n+2m,k}}(\si^{(n)})$ would be equal
to $\si^{(n+m)}$.  Now we take the approximation errors into account.
Suppose we have the estimates (\ref{kappa}) and (\ref{lambda}).
We then have
\[
(1-\e)^2\a_n\be_m\mu_y < K_y^{n+m} 
<(1+\e)^2\a_n\be_m\mu_y,\quad
\hbox{for all $y$}
\]
and then 
\[
\frac{1-\e}{(1+\e)^2}
\frac{\mu_x}{\sqrt w \be_m \mu_y} < \frac{K_x^n}{K_y^{n+m}} 
<\frac{1+\e}{(1-\e)^2}\frac{\mu_x}{\sqrt w \be_m \mu_y},\quad
\hbox{for all $x,y$}.
\]
We further have
\[
\frac{(1-\e)^2}{(1+\e)^2}\frac{\mu_x d_a}{\mu_y w} <
\frac{K_x^n}{K_y^{n+m}} L_a^m <
\frac{(1+\e)^2}{(1-\e)^2}\frac{\mu_x d_a}{\mu_y w},
\]
which means
\[
(1-5\e)\frac{\mu_x d_a}{\mu_y w} <
\frac{K_x^n}{K_y^{n+m}} L_a^m <
(1+5\e)\frac{\mu_x d_a}{\mu_y w},
\]
This shows 
\[
\|E_{A'_{2n+2m,0}\cap A_{2n+2m,k}}(\si^{(n)})-\si^{(n+m)}\|_2
\le 5C^5\e \|\si^{(n+m)}\|_2
\le 6C^5\e \|\si\|_{\st,2}.
\]
This is because we have
\[
(1-\e)\|\si\|_{\st,2}
<\|\si^{(n)}\|_2
<(1+\e) \|\si\|_{\st,2}
\]
and
\[
(1-\e)\|\si\|_{\st,2}
<\|\si^{(n+m)}\|_2
<(1+\e) \|\si\|_{\st,2},
\]
since the trace value of the minimal central projection
corresponding to the vertex $x$ in $A_{2n,0}$ is equal to
$\alpha_n \kappa_n^x \gamma_2^{-2n} \mu_x$ while
the trace value of the central projection 
corresponding to the vertex $x$ in 
$\bigoplus_{x\in V_0} \Str_x^k(\G)$ is 
$\displaystyle\frac{\mu_x^2}{w}$ and we have
(\ref{gamma}) and (\ref{kappa}).  We now have
\[
|\|E_{A'_{2n+2m,0}\cap A_{2n+2m,k}}(\si^{(n)})\|_2-\|\si^{(n)}\|_2|
< 8C^5 \e \|\si\|_{\st,2}.
\]
By Lemma \ref{L2}, we then have
\[
\|E_{A'_{2n+2m,0}\cap A_{2n+2m,k}}(\si^{(n)})-\si^{(n)}\|_2
< 4C^3 \sqrt\e \|\si\|_{\st,2}.
\]
We now have
\[
\|\si^{(n+m)}-\si^{(n)}\|_2\le
10C^5 \sqrt\e \|\si\|_{\st,2}.
\]

We first choose $n_1$ so that we have
(\ref{gamma}) and (\ref{kappa}) with $n=n_1$ and 
$\e=\displaystyle\frac{1}{100 C^{10}\cdot 4}$.
Starting with $l=1$, we make the following
procedure inductively.
We choose $m_l$ so that we have
(\ref{lambda}) with $m=m_l$ and 
$\e=\displaystyle\frac{1}{100 C^{10}\cdot 4^l}$
and (\ref{kappa}) and (\ref{kappa}) with $n=n_l+m_l$ and 
$\e=\displaystyle\frac{1}{100 C^{10}\cdot 4^{l+1}}$.
(Note that $\alpha_p \gamma_2^{-2p}\to 
\displaystyle\frac{1}{\sqrt w}$ as
$p\to\infty$ because we have
$\alpha_p \kappa_p^x \gamma_2^{-2p}\mu_x\to
\displaystyle\frac{\mu_x^2}{w}$ and
$\kappa_p^x\to \displaystyle\frac{\mu_x}{\sqrt w}$
as $p\to\infty$.)
We next set $n_{l+1}=n_l+m_l$.

Then we have
\[
\|\si^{(n_l)}-\si^{(n_{l+1})}\|_2\le
\frac{1}{2^l}\|\si\|_{\st,2}
\]
for all $\sigma$.  Because of this estimate, we know that the sequence
$\{\si^{(n_l)}\}_l$ converges in $A_{\infty,k}$ in the
strong operator topology for all $\sigma$.
We set $\Gamma(\bigoplus_{x\in V_0} \mu_x \si_x)
=\lim_{l\to\infty}\si^{(n_l)}.$  
Since $\si^{(n_l)}\in A'_{n_l,0}\cap A_{\infty,k}$,
we have $\Gamma(\bigoplus_{x\in V_0} \mu_x \si_x)\in
A'_{\infty,0}\cap A_{\infty,k}$.
This $\Gamma$ is clearly a linear map.  
We have $\|\Gamma(\bigoplus_{x\in V_0} \mu_x \si_x)\|_2
=\|\bigoplus_{x\in V_0} \si_x\|_{\st,2}$, so $\Gamma$
is injective.  This shows 
the dimension of the range of $\tilde P^k$ is
smaller than or equal to 
$\dim (A'_{\infty,0}\cap A_{\infty,k})$.
We thus conclude that the map $\Delta$ constructed
above is a linear isomorphism.
(This actually shows that $\bigoplus_{x\in V_0} \si_x$ is a flat
field and all $\sigma^{(n)}$ are equal in $A_{\infty,k}$.)
\end{proof}

\begin{remark}{\rm
The range of the projector matrix product operator of
length $k$ has obvious invariance under  rotation of
$2\pi/k$.  This passes to invariance of flat fields
of strings of length $k$ under rotation of $2\pi/k$.
Such invariance was observed by Ocneanu in early days
of the theory and this rotation was called a Fourier
transform of a flat field of strings.
See \cite{Li} for a recent progress of this notion of
the Fourier transform.
}\end{remark}

\begin{remark}{\rm
Replace the initial bi-unitary connection $W$ with $W'$.
The resulting subfactor $A_{0,\infty}\subset A_{1,\infty}$ does
not change, so the set $\{a,b,\dots\}$ of labels of the
irreducible bi-unitary connections does not change, but the
subfactor $A_{\infty,0}\subset A_{\infty,1}$ changes to its
dual subfactor.  So the range of the projector matrix product
operator also changes from the higher relative commutant to
the dual higher relative commutant of a subfactor in this process.
}\end{remark}

\begin{remark}{\rm
Recall that the Drinfel$'$d center of the fusion category of
$A_{0,\infty}$-$A_{0,\infty}$ bimodules arising from
the subfactor $A_{0,\infty}\subset A_{1,\infty}$ is a modular tensor
category related to the 2-dimensional topological order
appearing in \cite[Section 5]{BMWSHV}, as shown in \cite[Theorem 3.2]{K3}.
Note that the higher relative commutants of the {\em other}
subfactor $A_{\infty,0}\subset A_{\infty,1}$ appear here in
this paper.  Relations between these two subfactors are
clarified in \cite[Theorem 3.3]{S2}.
}\end{remark}

\begin{remark}{\rm
The range of  $P^k$ does not have  a natural algebra structure, but
we know from the above Theorem that it has a natural structure of
a $*$-algebra.
}\end{remark}

\begin{example}{\rm
An almost trivial example is given as follows.
All the sets $V_0,V_1,V_2,V_3$ are one-point sets and identified
with $\{x\}$.  All the graphs $\G,\G',\H,\H'$ consist of
$d$ multiple edges from $x$ to $x$ and they are all identified.
We have $\mu_x=1$, $\gamma_1=\gamma_2=d$ and
the connection $W$ is given as in Figure \ref{trivial}.

\begin{figure}[h]
\begin{center}
\begin{tikzpicture}
\draw [thick, ->] (1,1)--(2,1);
\draw [thick, ->] (1,2)--(2,2);
\draw [thick, ->] (1,2)--(1,1);
\draw [thick, ->] (2,2)--(2,1);
\draw (1.5,1.5)node{$W$};
\draw (1,1.5)node[left]{$\xi_0$};
\draw (2,1.5)node[right]{$\xi_2$};
\draw (1.5,1)node[below]{$\xi_1$};
\draw (1.5,2)node[above]{$\xi_3$};
\draw (3.7,1.5)node{$=\delta_{\xi_0,\xi_2}\delta_{\xi_1,\xi_3}.$};
\end{tikzpicture}
\caption{An almost trivial example}
\label{trivial}
\end{center}
\end{figure}

This is a flat connection, and ths set $V$ is identified with
$\{x\}$.  We have $d_x=1$ and $w=1$.  

In this case, the range of $\tilde P^k$ is $M_d(\C)^{\otimes k}$,
where $M_d(\C)$ is the $d\times d$ full matrix algebra with
complex entries.

In this example, the natural $C^*$-algebra appearing in
the inductive limit is a UHF algebra.  In the general case,
we have an AF algebra instead.
}\end{example}

\begin{example}{\rm
An easy example of a bi-unitary connection arises from a finite
group $G$ as in \cite[Figure 10.25]{EK}.  This corresponds to
a trivial 3-cocycle case considered in \cite[Section 6]{BMWSHV}.

We have $\gamma_1=\gamma_2=\sqrt{|G|}$ and
this is a flat connection.  The sets $V_0$ and $V$ are both
identified with $G$ as sets. All $d_a$ are 1 and $w=|G|$.
}\end{example}

\begin{example}{\rm
Consider the example arising from the Dynkin diagram $A_n$
as in Example \ref{dynkin}.
This is a flat connection and the set $V$ is identified
with $V_0$ which consists of $[(n+1)/2]$ vertices.  The value
$w$ is equal to $\displaystyle\frac{n+1}{4\sin^2\frac{\pi}{n+1}}$
\cite{Ch}.  In this case, the range of $\tilde P^k$ is generated
by the Jones projections $e_1,e_2,\dots,e_{k-1}$, where the
Jones projections are given as in \cite[Definition 11.5]{EK}.
}\end{example}

\begin{remark}\label{R1}{\rm
Since our treatments include non-flat bi-unitary connections, our
setting looks more general than that in \cite{LFHSV}.  That is,
it seems our 4-tensors give a larger class than those covered in
\cite{LFHSV}, but 
exact relations between ours and theirs are not clearly understood.
It would be interesting to clarify this issue.
}\end{remark}

\end{document}